\documentclass[a4paper,11pt,reqno]{amsart}

\usepackage{syntonly}

\usepackage{graphicx}
\usepackage{amsmath, amssymb, amsthm, textcomp}
\usepackage{cite}
\usepackage{fullpage}

\newcommand{\Stir}[2]{\genfrac{\{}{\}}{0pt}{}{#1}{#2}}
\newcommand{\RR}{\mathbb{R}}
\newcommand{\rt}{\mathrm{root}}
\newcommand{\inv}{\mathrm{inv}}
\newcommand{\Z}{\mathrm{Z}}
\newcommand{\V}{\mathrm{V}}
\newcommand{\D}{\mathrm{D}}
\renewcommand{\H}{\text{\normalfont H}}
\newcommand{\DZ}{\text{\normalfont D}_z}
\newcommand{\DQ}{\text{\normalfont D}_q}
\newcommand{\Arg}{\text{\normalfont Arg}}
\newcommand{\EE}[1]{\mathbb{E} \left( #1 \right) }
\newcommand{\PP}[1]{\mathbb{P} \left\{ #1 \right\} }
\newcommand{\ff}[1]{\underline{#1}}
\newcommand{\rf}[1]{\overline{#1}}
\newcommand{\OO}[1]{\mathcal{O} \left(#1\right)}
\newcommand{\oneplusOO}[1]{\left(1 + \OO{#1}\right)}
\newcommand{\convd}{\xrightarrow{(d)}}

\newcommand{\CC}{\mathbb{C}}

\newcommand{\ii}{\mathrm{i}}
\renewcommand{\gcd}[1]{\mathrm{gcd} \lbrace #1 \rbrace}
\renewcommand{\mod}{\,\mathrm{mod}\,}
\newcommand{\NN}{\mathbb{N}}
\newcommand{\T}{\mathcal{T}}
\newcommand{\hT}{\hat{\mathcal{T}}}
\newcommand{\set}[1]{\left\{ #1 \right\}}

\newcommand{\dcup}{\dot{\cup}}
\newcommand{\hI}{\hat{I}}
\newcommand{\nodeOne}{\text{\textbigcircle}\hspace*{-0.29cm}1\,}
\newtheorem{theorem}{Theorem}[section]
\newtheorem{lemma}[theorem]{Lemma}
\newtheorem{definition}[theorem]{Definition}
\newtheorem{prop}[theorem]{Proposition}

\begin{document}

\author[A.~Panholzer]{Alois Panholzer}
\address{Alois Panholzer\\
Institut f{\"u}r Diskrete Mathematik und Geometrie\\
Technische Universit\"at Wien\\
Wiedner Hauptstr. 8-10/104\\
A-1040 Wien, Austria}
\email{Alois.Panholzer@tuwien.ac.at}

\author[G.~Seitz]{Georg Seitz}
\address{Georg Seitz\\
Institut f{\"u}r Diskrete Mathematik und Geometrie\\
Technische Universit\"at Wien\\
Wiedner Hauptstr. 8-10/104\\
A-1040 Wien, Austria}
\email{Georg.Seitz@tuwien.ac.at}

\thanks{This work has been supported by the Austrian Science Foundation FWF, grant S9608-N23.}

\title[The number of inversions in labelled tree families]{Limiting distributions for the number of inversions in labelled tree families}

\begin{abstract}
  We consider so-called simple families of labelled trees, which contain, e.g., ordered, unordered, binary and cyclic labelled trees as special instances, and study the global and local behaviour of the number of inversions. In particular we obtain limiting distribution results for the total number of inversions as well as the number of inversions induced by the node labelled $j$ in a random tree of size $n$.
\end{abstract}

\date{\today}

\keywords{inversions, simply generated trees, limiting distributions}

\subjclass[2000]{05C05, 60F05, 05A16}

\maketitle

\section{Introduction}

Throughout this paper we always consider \emph{rooted trees} $T$ in which the vertices are \emph{labelled} with distinct integers of $\{1, \ldots, |T|\}$,  where $|T|$ 
is the size (i.e., the number of vertices) of $T$.
An \emph{inversion} in a tree $T$ is a pair $(i,j)$ of vertices (we may always identify a vertex with its label), such that $i > j$ and $i$ lies on the unique
path from the root node $\rt(T)$ of $T$ to $j$ (thus $i$ is an ascendant of $j$ or, equivalently, $j$ is a descendant of $i$).
Let us denote by $\inv(T)$ the number of inversions in $T$.

In \cite{Gessel-Sagan-Yeh_TreeInversions,Mallows-Riordan_InversionEnumerator} studies concerning the number of inversions in some important combinatorial tree families $\mathcal{T}$ have been given by introducing 
so-called \emph{tree inversion polynomials}. They shall be defined as follows\footnote{Later we consider weighted tree families and introduce extensions of this and further definitions.}:
\begin{equation*}
  J_{n}(q) := \sum_{T \in \mathcal{T}: |T|=n} q^{\inv(T)}.
\end{equation*}
Actually, unlike in our studies, in \cite{Gessel-Sagan-Yeh_TreeInversions,Mallows-Riordan_InversionEnumerator} the authors exclusively considered trees with the root node labelled $1$.
Thus, in order to avoid confusion, we introduce also the slightly modified polynomials $\hat{J}_{n}(q) := \sum_{\substack{T \in \mathcal{T}: \\ |T|=n \; \text{and} \; \rt(T)=1}} q^{\inv(T)}$.
For \emph{unordered trees}, i.e., trees, where one assumes that to each vertex there is attached a (possibly empty) set of children 
(thus there is no left-to-right ordering of the children of any node), Mallows and Riordan \cite{Mallows-Riordan_InversionEnumerator} could give an explicit formula for a suitable generating function of 
the corresponding tree inversion polynomials:
\begin{equation*}
  \exp\left(\sum_{n \ge 1} (q-1)^{n-1} \hat{J}_{n}(q) \frac{t^{n}}{n!}\right) = \sum_{n \ge 0} q^{\binom{n}{2}} \frac{t^{n}}{n!}.
\end{equation*}
Gessel et al. \cite{Gessel-Sagan-Yeh_TreeInversions} considered $\hat{J}_{n}(q)$ for three other tree families:
\begin{itemize}
\item \emph{Ordered trees:} one assumes that to each vertex there is attached a (possibly empty) sequence of children 
(thus there is a left-to-right ordering of the children of each node).
\item \emph{Cyclic trees:} ordered trees, where one assumes that 
cyclic rearrangements of the subtrees of any node give the same tree.
\item \emph{Plane trees:} ordered trees, where one assumes that cyclic rearrangements of the subtrees of the root node give the same tree.
\end{itemize}
Unlike for unordered trees, no explicit formul{\ae} for a suitable generating function of the tree inversion polynomial of the latter tree families could be given, but the authors provide
exact and asymptotic results for the evaluations of $\hat{J}_{n}(q)$ for the specific values $q=0, 1, -1$. In particular, $\hat{J}_{n}(0)$ enumerates so-called \emph{increasing trees},
i.e., trees, where each child node has a label larger than its parent node.

Besides these studies it seems natural to ask, for a given combinatorial family $\mathcal{T}$ of trees, questions about the ``typical behaviour'' of the number of inversions in a tree $T \in \mathcal{T}$
of size $n$. In a probabilistic setting we may introduce a random variable $I_{n}$, which counts the number of inversions of a \emph{random tree} of size $n$, i.e., a tree chosen uniformly at random
from all trees of the family $\mathcal{T}$ of size $n$. Of course, this more probabilistic point of view and the before-mentioned combinatorial approach are closely related. Let us denote by $T_{n}$ the number of trees
of $\mathcal{T}$ of size $n$. Then it holds 
\begin{equation*}
   J_{n}(q) = T_{n} \sum_{k \ge 0} \mathbb{P}\{I_{n}=k\} q^{k},
\end{equation*}
i.e., the probability generating function $p_{n}(q) := \sum_{k \ge 0} \mathbb{P}\{I_{n}=k\} q^{k}$ of the random variable $I_{n}$ is simply given by $p_{n}(q) = \frac{J_{n}(q)}{T_{n}}$, 
and it holds $T_{n,k} = T_{n} \mathbb{P}\{I_{n}=k\}$ for the number $T_{n,k} := [q^{k}] J_{n}(q)$ of trees of size $n$ with exactly $k$ inversions.

A main concern of this paper is to describe the asymptotic behaviour of the random variable $I_{n}$ for various important tree families by proving limiting distribution results.
In our studies of $I_{n}$ we use as tree-models so-called \emph{simply generated tree families} \cite{Flajolet-Sedgewick_AnalyticCombinatorics,Meir-Moon}, which contain many important combinatorial tree families,
such as the before-mentioned unordered trees, ordered trees, and cyclic trees, but also others such as, e.g., binary trees, $t$-ary trees, and Motzkin trees, as special instances.
Simply generated trees are weighted ordered trees, where, given a degree-weight sequence, each node gets a weight according to its out-degree, i.e., the number of its children 
(see Subsection~\ref{ssec21} for a precise definition); we remark that in probability theory such tree models are known as Galton-Watson trees.
As a main result we can show, provided the degree-weight sequence satisfies certain mild growth conditions (which are all satisfied for the before-mentioned tree families),
that, after a suitable normalization of order $n^{\frac{3}{2}}$, $I_{n}$ converges in distribution to a distribution known as the Airy distribution
(see Subsection~\ref{ssec32}). We remark that the Airy distribution also appears in enumerative studies 
of other combinatorial objects such as, e.g., the area below lattice paths \cite{Louchard_BrownianExcursion}, the area of staircase polygons \cite{Schwerdtfeger-Richard-Thatte_StaircasePolygons}, sums of parking functions \cite{Kung-Yan_SumsParkingFunctions}, and the costs of linear probing hashing algorithms \cite{Flajolet-Poblete-Viola_LinearProbingHashing}.
For the particular tree family of unordered trees this limiting distribution result for $I_{n}$ has been shown already by Flajolet et al.\ in \cite{Flajolet-Poblete-Viola_LinearProbingHashing}
during their analysis of a linear probing hashing algorithm by using close relations between the insertion costs of this algorithm and the number of inversions in unordered trees.
We note that we show convergence in distribution, thus obtaining asymptotic results for $\mathbb{P}\{I_{n} \le x n^{\frac{3}{2}}\}$ or alternatively for the sums $\sum_{k \le x n^{\frac{3}{2}}} T_{n,k}$, 
with $x \in \RR^{+}$, but we do not obtain local limit laws, i.e., results concerning the behaviour of the probabilities $\mathbb{P}\{I_{n} = k\}$ or the numbers $T_{n,k}$ itself.

Besides this ``global study'' of the number of inversions in a random tree we are additionally interested in the contribution to this quantity induced by a specific label $j$, i.e., in a ``local study''. 
To do this we introduce random variables $I_{n,j}$, which count the number of inversions of the kind $(i,j)$, with $i > j$ an ancestor of $j$, in a random tree of size $n$.
Of course, one could also introduce ``local inversion polynomials'' $J_{n,j}(q) := T_{n} \sum_{k \ge 0} \mathbb{P}\{I_{n,j}=k\} q^{k}$. Note that $I_{n} = \sum_{j=1}^{n} I_{n,j}$, but the random variables $I_{n,j}$ are highly dependent.
In our studies we describe the asymptotic behaviour of the random variable $I_{n,j}$, depending on the growth of $j=j(n)$ with respect to $n$. In particular, we obtain that for the main portion of labels, i.e., for $j \ll n-\sqrt{n}$, $I_{n,j}$ converges, after suitable normalization of order $\sqrt{n}$, in distribution to a Rayleigh distribution. We remark that the Rayleigh distribution also appears frequently when studying combinatorial objects, see, e.g., \cite{Flajolet-Sedgewick_AnalyticCombinatorics}. If $n-j \sim \rho \sqrt{n}$ or $n-j = o(\sqrt{n})$ then the behaviour changes.
Apart from asymptotic results, we can for two particular tree families, namely ordered and unordered trees, also give explicit formul{\ae} for the probabilities $\mathbb{P}\{I_{n,j}=k\}$.
An example of a labelled tree and the parameters studied is given in Figure~\ref{fig:Example}
\begin{figure}
\begin{center}
  \includegraphics[height=3.5cm]{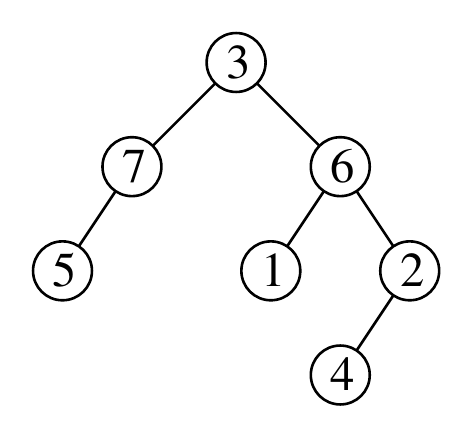}
\end{center}
\caption{A binary labelled tree of size $7$ with a total number of $6$ inversions, namely $(3,1)$, $(6,1)$, $(3,2)$, $(6,2)$, $(6,4)$, $(7,5)$. Thus two inversions each are induced by the nodes
$1$ and $2$, whereas one inversion each is induced by the nodes $4$ and $5$.\label{fig:Example}}
\end{figure}

We remark that the asymptotic results obtained for inversions in trees are completely different from the corresponding ones for permutations of a set $\{1, 2, \dots, n\}$.
It is well-known (see, e.g., \cite{Louchard-Prodinger_InversionsPermutations}) that the total number of inversions in a random permutation of size $n$ is asymptotically normal distributed with expectation and variance of order $n^{2}$ and $n^{3}$, respectively. Trivially, the number of inversions of the kind $(i,j)$, with $i > j$ an element to the left of $j$, in a random permutation of size $n$ is uniformly distributed on $\{0, 1, \dots, n-j\}$.

\smallskip

The plan of this paper is as follows. In Section~\ref{sec2} we collect definitions and known results about simply generated tree families,
whereas in Section~\ref{sec3} we state the main results of this paper concerning the random variables $I_{n}$ and $I_{n,j}$.
A proof of the results for $I_{n}$ and $I_{n,j}$ is given in Section~\ref{sec4} and Section~\ref{sec5}, respectively.

\smallskip

Before continuing we define some notation used throughout this paper.
The operator $[z^n]$ extracts the coefficient of $z^n$ from a power series 
$A(z) = \sum_{n \geq 0} a_n z^n$, i.e., $[z^n] A(z) = a_n$. For $s \in \NN_0$, $x^{\ff{s}}$ 
(resp.\ $x^{\rf{s}}$) denotes the $s$-th falling (resp.\ rising) factorial of $x$, i.e., 
$x^{\ff{0}} = x^{\rf{0}} = 1$, and $x^{\ff{s}} = x(x-1) \cdots (x-s+1)$, 
$x^{\rf{s}} = x(x+1) \cdots (x+s-1)$, for $s \geq 1$.
Furthermore, for each variable $x$ we denote the differential operator with respect to $x$
by $\D_x$, and we define two operators $\V$ ($= \V_{q}$) and $\Z$, which act on bivariate power series $G(z,q)$ by $\V G(z,q) := G(z,1)$, and $\Z G(z,q) := z G(z,q)$ (analogous definitions
for multivariate power series).
Moreover, if $X$ and $X_n$, $n \geq 1$, are random variables, then $X_n \convd X$ denotes
the weak convergence (i.e., the convergence in distribution) of the sequence $(X_n)_{n \geq 1}$
to $X$. 

\section{Labelled families of simply generated trees and auxiliary results\label{sec2}}

Families of simply generated trees were introduced by Meir and Moon in \cite{Meir-Moon}. 
As mentioned before, many important combinatorial tree families such as, e.g., labelled unordered trees (also called Cayley trees), 
binary trees, labelled cyclic trees (also called mobile trees) and ordered trees (also called planted plane trees), can be considered as special instances of simply generated trees.

We now recall how simply generated tree families are defined in the labelled context, and then 
collect some well-known auxiliary results. Note that in the following the term ``tree'' will always denote a labelled tree.

\subsection{Definitions\label{ssec21}}
A class $\T$ of (labelled) simply generated trees is defined in the following way:
One chooses a sequence $(\varphi_{\ell})_{\ell \geq 0}$ (the so-called \emph{degree-weight sequence}) of nonnegative real numbers with $\varphi_0 > 0$. 
Using this sequence, the \textit{weight $w(T)$} of each ordered tree (i.e., each rooted tree, 
in which the children of each node are ordered from left to right) is defined by 
$w(T) := \prod_{v \in T} \varphi_{d(v)}$, where by $v \in T$ we mean that $v$ is a vertex of $T$ and 
$d(v)$ denotes the number of children of $v$ (i.e., the out-degree of $v$). The family $\T$ associated to the degree-weight sequence $(\varphi_{\ell})_{\ell \geq 0}$ 
then consists of all trees $T$ (or all trees $T$ with $w(T) \neq 0$) together with their weights.

We let $T_n := \sum_{|T| = n} w(T)$ denote the the \emph{total weight} of all trees of size $n$ in $\T$,
and define by $T(z) := T_n \frac{z^n}{n!}$ its exponential generating function. Then it follows that $T(z)$ satisfies the (formal)
functional equation
\begin{equation}\label{eqn:T}
 T(z) = z \varphi(T(z)),
\end{equation}
where the degree-weight generating function $\varphi(t)$ is defined via $\varphi(t) := \sum_{\ell \geq 0} \varphi_{\ell} t^{\ell}$. 

We want to remark that each simply generated tree family $\T$ can also be defined by a formal 
equation of the form
\begin{equation}\label{eqn:T-comb}
 \T = \bigcirc * \varphi(\T),
\end{equation}
where $\bigcirc$ denotes a node, $*$ is the combinatorial product of labelled objects, and 
$\varphi(\T)$ is a certain substituted structure (see, e.g., \cite{Flajolet-Sedgewick_AnalyticCombinatorics}). 
Hence, the functional equation \eqref{eqn:T} can be obtained directly from the combinatorial construction of $\T$ 
using the symbolic method (cf. \cite{Flajolet-Sedgewick_AnalyticCombinatorics}). Furthermore, $(T_n)_{n \geq 1}$ is for many important simply 
generated tree families a sequence of natural numbers, and then the total weight $T_n$ can be interpreted as the 
number of trees of size $n$ in $\T$. We now give several examples where this is the case.

\subsection*{Examples:}
\begin{itemize}
 \item Binary trees can be defined combinatorially as follows:
    \[
     \T = \bigcirc * (\lbrace\square\rbrace \:\dcup\: \T) * (\lbrace\square\rbrace \:\dcup\: \T).
    \]
    Here, $\square$ denotes an empty subtree and $\dcup$ is the disjoint union. This formal 
    equation expresses that each binary tree consists of a root node and a left and a right subtree, 
    each of which is either a binary tree or empty. The formal equation for $\T$ can directly 
    be translated into a functional equation for $T(z)$, namely
    \[
     T(z) = z (1+T(z))^2.
    \]
    Hence, binary trees are the simply generated tree family defined by $\varphi(t) = (1+t)^2$, i.e., by the degree-weight sequence $\varphi_{\ell} = \binom{2}{\ell}$, $\ell \ge 0$.
 \item Ordered trees are rooted trees, in which the children of each node are ordered. Thus,
   combinatorially speaking, each ordered tree consists of a root node and a sequence of ordered trees, 
   \[
    \T = \bigcirc * \text{\textsc{Seq}}(\T) = \bigcirc * \big(\lbrace\square\rbrace \:\dcup\: \T \:\dcup\: 
         \T^2 \:\dcup\: \T^3 \:\dcup\: \ldots\big).
   \]
   From this one gets the functional equation
   \[
    T(z) = \frac{z}{1 - T(z)},
   \]
   i.e., $\varphi(t) = \frac{1}{1-t}$. Of course, this corresponds to the degree-weight sequence $\varphi_{\ell} = 1$, $\ell \ge 0$.
 \item Unordered trees are rooted trees in which there is no order on the children of any node. 
   Hence, each unordered tree consists of a root node and a set of unordered trees, which can be 
   written formally as
   \[
    \T = \bigcirc * \text{\textsc{Set}}(\T) = \bigcirc * \Big(\lbrace\square\rbrace \:\dcup\: \T \:\dcup\: 
         \frac{\T^2}{2!} \:\dcup\: \frac{\T^3}{3!} \:\dcup\: \ldots\Big).
   \]
   This leads to the functional equation
   \[
    T(z) = z \exp(T(z)),
   \]
   i.e., one has $\varphi(t) = \exp(t)$, or equivalently $\varphi_{\ell} = 1/{\ell}!$, $\ell \ge 0$.
 \item Cyclic trees may be considered as equivalence classes of ordered trees, where cyclic rearrangements of the subtrees of nodes lead to a tree of the same class.
   Hence, each cyclic tree is either a single root node or it consists of a root node and a (non-empty) cycle of unordered trees, which can be 
   written formally as
   \[
    \T = \bigcirc \:\dcup\: \bigcirc * \:\text{\textsc{Cyc}}(\T) = \bigcirc * \big(\lbrace\square\rbrace \:\dcup\: \text{\textsc{Cyc}}(\T)\big).
   \]
   This leads to the functional equation
   \[
    T(z) = z \left(1+\log\left(\frac{1}{1-T(z)}\right)\right),
   \]
   i.e., one has $\varphi(t) = 1+\log\big(\frac{1}{1-t}\big)$, or equivalently $\varphi_{0}=1$ and $\varphi_{\ell} = 1/{\ell}$, $\ell \ge 1$.
\end{itemize}
We remark that plane trees as considered in \cite{Gessel-Sagan-Yeh_TreeInversions} are not covered by the definition of simply generated trees. However, since every subtree of the root node of a plane tree is an ordered tree,
the methods applied in this work for a study of the number of inversions can be adapted easily to treat also this tree family, which leads to the same limiting distribution results as for ordered trees. Thus we omit computations for this tree family.

\subsection{Auxiliary results\label{ssec22}}
We now collect some known results (see, e.g., \cite{Flajolet-Sedgewick_AnalyticCombinatorics,Panholzer_AncestorTree}) on the function $T(z)$ satisfying 
\eqref{eqn:T}.
First note that in general $T(z)$ and $\varphi(t)$ must be regarded as formal power 
series, because they do not need to have a positive radius of convergence, and then \eqref{eqn:T} 
must be understood as a formal equation. Thus, in order to analyze properties of simply generated 
tree families by analytic methods, we will need to make certain assumptions on $\varphi$. In particular, 
we will assume that $\varphi(t)$ has a positive radius of convergence $R$, and that there exists a 
minimal positive solution $\tau < R$ of the equation 
\begin{equation}\label{eqn:tphi'}
 t \varphi'(t) = \varphi(t).
\end{equation}
If we define 
\begin{equation}\label{eqn:d}
 d := \gcd{\ell: \varphi_{\ell} > 0},
\end{equation}
it then follows that \eqref{eqn:tphi'} has exactly $d$ 
solutions of smallest modulus, which are given by $\tau_j = \omega^j \tau$, for $0 \leq j \leq d-1$, 
where $\omega = \exp (\frac{2\pi \ii}{d})$. 
From the implicit function theorem it follows that the equation $z = \frac{t}{\varphi(t)}$ is not 
invertible in any neighbourhood of $t = \tau_j$, for $0 \leq j \leq d-1$. 
This leads to $d$ dominant singularities of $T(z)$ at $z = \rho_j$, where $\rho_j = \omega^j \rho$, 
$\rho = \frac{\tau}{\varphi(\tau)}$. 

For our purpose, it is important to note that under the above assumptions, $T(z)$ is amenable to
singularity analysis (cf. \cite{Flajolet-Odlyzko_SingularityAnalysis}), i.e., there are constants 
$\eta > 0$ and $0 < \phi < \pi/2$ such that $T(z)$ is analytic in the domain 
$\set{z \in \CC : |z| < \rho + \eta, z \neq \rho_j, |\Arg(z-\rho_j)| > \phi, 
\text{ for all } 0 \leq j \leq d-1}$. 
The local expansion of $T(z)$ around the singularity $z = \rho_j$ is given by
\begin{equation}\label{eqn:T_exp}
 T(z) = \tau_j - \omega^j \sqrt{\frac{2 \varphi(\tau)}{\varphi''(\tau)}} \sqrt{1 - \frac{z}{\rho_j}}
        + \OO{\rho_j -z}.
\end{equation}
Using singularity analysis and summing up the contributions of the $d$ dominant singularities,
one obtains
\begin{equation}\label{eqn:T_n}
 \frac{T_n}{n!} = [z^n] T(z) = \frac{d \sqrt{\varphi(\tau)}}{\sqrt{2 \pi \varphi''(\tau)} \rho^n n^{\frac{3}{2}}} 
    \oneplusOO{\frac{1}{n}}, 
\end{equation}
for $n \equiv 1 \mod d$. If $n \not\equiv 1 \mod d$, one has of course $T_n = 0$, because in this
case each ordered tree of size $n$ has weight zero.

In our analysis, we will further make use of the functions $\varphi^{(m)}(T(z))$ 
(where $\varphi^{(m)}(t)$ is the $m$-th derivative of $\varphi(t)$). 
Each of these functions has $d$ dominant singularities at $z = \rho_j$, $0 \leq j \leq d-1$, and 
complies with the requirements for singularity analysis. 
Around $z = \rho_j$, one has the expansion
\begin{equation}\label{eqn:phi_k}
\varphi^{(m)}(T(z)) = \varphi^{(m)}(\tau_j) - \varphi^{(m+1)}(\tau_j) \omega^j 
    \sqrt{\frac{2 \tau}{\rho \varphi''(\tau)}} \sqrt{1-\frac{z}{\rho_j}} + \OO{\rho_j - z},
\end{equation}
and we will especially make use of the expansion
\begin{equation}\label{eqn:zphi_exp}
z \varphi'(T(z)) = 1 - \sqrt{2 \rho \tau \varphi''(\tau)} \sqrt{1-\frac{z}{\rho_j}} + \OO{\rho_j - z}.
\end{equation}

\section{Parameters studied and results\label{sec3}}

\subsection{Parameters studied}

Consider a simply generated tree family $\T$ associated to a degree-weight sequence $(\varphi_{\ell})_{\ell \ge 0}$.
In our analysis of parameters in trees of $\T$ we will always use the ``random tree model for weighted trees'', i.e.,
when speaking about a random tree of size $n$ we assume that each tree $T$ in $\T$ of size $n$ is chosen with a probability proportional to its weight $w(T)$.

The main quantities of interest are the random variable $I_{n}$, which counts the total number of inversions of a random simply generated tree of size $n$,
and the random variable $I_{n,j}$, which counts the number of inversions of the kind $(i,j)$, with $i > j$ an ancestor of $j$, in a random simply generated tree of size $n$.

We mention the relation to a suitably adapted tree inversion polynomial for weighted tree families:
\begin{equation*}
  J_{n}(q) := \sum_{T \in \T: |T|=n} w(T) \cdot q^{\inv(T)} = T_{n} \sum_{k \ge 0} \mathbb{P}\{I_{n}=k\} q^{k}.
\end{equation*}

\subsection{Auxiliary results for probability distributions\label{ssec32}}

We collect some basic facts about two important probability distributions appearing later in our analysis.

\begin{definition}\label{def:airy}
The Airy distribution is the distribution of a random variable $I$ with $r$-th moments
 \[
\mu_r := \EE{I^r} = \frac{2 \sqrt{\pi}}{\Gamma(\frac{3r-1}{2})} C_r,
\]
where the constants $C_r$ can inductively be defined by 
\begin{equation}\label{eqn:C_r}
2 C_r = (3r-4)rC_{r-1} + \sum_{j=1}^{r-1} \binom{r}{j} C_j C_{r-j}, \quad r \geq 2, \quad C_1 = \frac{1}{2}. 
\end{equation}
\end{definition}
Since we will use the method of moments in order to establish our results, the following well-known result about the Airy distribution is important
(see, e.g., \cite{Flajolet-Louchard_AiryDistribution}, where one can find more details
about the Airy distribution and some equivalent definitions).
\begin{lemma}\label{lem:airy}
The Airy distribution is uniquely determined by its sequence of moments $(\mu_r)_{r \geq 1}$.
\end{lemma}

\begin{definition}\label{def:rayleigh}
The Rayleigh distribution with parameter $\sigma > 0$ is the distribution of a random variable $X_{\sigma}$ with probability density function
\begin{equation}
   f_{\sigma}(x) = \frac{x}{\sigma^{2}} \, e^{-\frac{x^{2}}{2\sigma^{2}}}, \quad x > 0.
\end{equation}
\end{definition}

The following basic fact about the Rayleigh distribution will be required in our analysis.
\begin{lemma}\label{lem:rayleigh}
The Rayleigh distribution is uniquely determined by its sequence of $r$-th moments $(\mu_r)_{r \geq 1}$, which are given as follows:
\begin{equation}
  \mu_r := \EE{X_{\sigma}^r} = \sigma^{r}\, 2^{\frac{r}{2}} \, \Gamma\big(1+\frac{r}{2}\big).
\end{equation}
\end{lemma}

\subsection{Results}

Let $\T$ be the labelled family of simply generated trees associated to a degree-weight sequence
$(\varphi_{\ell})_{\ell \geq 0}$, where the function $\varphi(t) := \sum_{\ell \geq 0} \varphi_{\ell} t^{\ell}$ has 
positive radius of convergence $R$, and equation \eqref{eqn:tphi'} has a minimal positive solution 
$\tau < R$. Furthermore, let $\rho = \frac{\tau}{\varphi(\tau)}$ (recall the definitions of Subsection~\ref{ssec22}).
Then the following holds:

\begin{theorem}[Global behaviour]\label{thm:global}
 The random variable $I_n$, which counts the total number of inversions in a random tree of size 
 $n$ of $\T$ is, after proper normalization, asymptotically Airy distributed:\\
 It holds that $\EE{I_n} \sim c_{\varphi} \sqrt{\pi} n^{3/2}$, where 
 $c_{\varphi} = \frac{1}{\sqrt{8 \rho \tau \varphi''(\tau)}}$, and
  \[
   \frac{I_n}{c_{\varphi} n^{3/2}} \convd I,
  \]
  where $I$ is an Airy distributed random variable.
\end{theorem}

\begin{theorem}[Local behaviour]\label{thm:local}
  The random variable $I_{n,j}$, which counts the number of inversions of the kind $(i,j)$, with $i>j$ an ancestor of $j$, in a random tree of size $n$ of $\T$
  has, depending on the growth of $1 \le j=j(n) \le n$, the following asymptotic behaviour.
  \begin{itemize}
    \item Region $n - j \gg \sqrt{n}$: $I_{n,j}$ is, after proper normalization, asymptotically Rayleigh distributed:
    \begin{equation*}
      \frac{\sqrt{n}}{n-j}I_{n,j} \convd X_{\sigma},
    \end{equation*}
    where $X_{\sigma}$ is a Rayleigh distributed random variable with parameter $\sigma:=\frac{1}{\sqrt{\rho \tau \varphi''(\tau)}}$.
    \item Region $n-j \sim \alpha \sqrt{n}$, with $\alpha \in \mathbb{R}^{+}$:
    $I_{n,j}$ converges in distribution to a discrete random variable $Y_{\gamma}$, with 
    \begin{equation*}
      \mathbb{P}\{Y_{\gamma}=k\} = \frac{\gamma^{k}}{k!} \int_{0}^{\infty} x^{k+1} e^{-\frac{x^{2}}{2} - \gamma x} dx, \quad k \ge 0,
    \end{equation*}
    and $\gamma:=\frac{\alpha}{\sqrt{\rho \tau \varphi''(\tau)}}$.
    \item Region $n-j \ll \sqrt{n}$: $I_{n,j}$ converges in distribution to a random variable with all its mass concentrated at $0$, i.e., $I_{n,j} \convd 0$.
  \end{itemize}
\end{theorem}

\subsection{Examples:\label{ssec34}}
Before we prove these results, we apply them to our example tree families:
\begin{itemize}
 \item Binary trees: From the equation $2t(t+1) = t \varphi'(t) = \varphi(t) = (t+1)^2$ we get
    the positive solution $\tau = 1$, and hence $c_{\varphi} = \frac{1}{2}$ and $\sigma = \sqrt{2}$. 
    Thus, if we let $I_n$ denote the number of inversions in a random binary tree of size $n$, 
    then $\frac{2 I_n}{n^{3/2}}$ converges in distribution to an Airy distributed random variable.
    Furthermore, for the number $I_{n,j}$ of inversions in a random binary tree of size $n$ induced by node $j$, it holds that
    $\frac{\sqrt{n}}{n-j}I_{n,j}$ converges, for $n-j \gg \sqrt{n}$, in distribution to a Rayleigh distributed random variable with parameter $\sqrt{2}$.
 \item Ordered trees: The equation $\frac{t}{(1-t)^2} = \frac{1}{1-t}$ yields $\tau = \frac{1}{2}$,
    and further $c_{\varphi} = \frac{1}{4}$ and $\sigma = \frac{1}{\sqrt{2}}$. Hence, for the number $I_n$ of inversions in a 
    random ordered tree of size $n$, it holds that $\frac{4 I_n}{n^{3/2}}$ is asymptotically Airy 
    distributed. Furthermore, the normalized number of inversions $\frac{\sqrt{n}}{n-j}I_{n,j}$ induced by node $j$, is, for $n-j \ll \sqrt{n}$,
    asymptotically Rayleigh distributed with parameter $1/\sqrt{2}$.
    
    We further note that for ordered trees the exact distribution of $I_{n,j}$ is given as follows (for $1 \le j \le n$ and $0 \le k \le n-j$):
    \begin{equation}\label{eqn:InjOrd}
       \mathbb{P}\{I_{n,j} = k\} = \frac{1}{\binom{n-1}{j-1}\binom{2(n-1)}{n-1}} \sum_{\ell=n-j-k}^{n-k-1} \binom{\ell}{n-j-k} \binom{2n-2}{\ell} \binom{n-\ell-1}{k} \frac{2n-1-2\ell}{2n-1-\ell}.
    \end{equation}
 \item Unordered trees: Here, one has $\tau = 1$ and thus $c_{\varphi} = \frac{1}{\sqrt{8}}$ and $\sigma = 1$.
    This shows that $\frac{\sqrt{8} I_n}{n^{3/2}}$ converges in distribution to an Airy 
    distributed random variable and that $\frac{\sqrt{n}}{n-j}I_{n,j}$ converges, for $n-j \ll \sqrt{n}$, in distribution to a Rayleigh distributed random variable with parameter $1$.
    
    Also for unordered trees the exact distribution of $I_{n,j}$ can be stated explicitly. It holds (for $1 \le j \le n$ and $0 \le k \le n-j$):
    \begin{equation}\label{eqn:InjUnord}
       \mathbb{P}\{I_{n,j} = k\} = \frac{(j-1)! (n-j)!}{n^{n-1}} \sum_{\ell=n-j-k}^{n-k-1} \binom{\ell}{n-j-k} \binom{n-\ell-1}{k} \frac{(n-\ell)n^{\ell-1}}{\ell!}.
    \end{equation}
 \item Cyclic trees: The positive real solution of the equation $\frac{t}{1-t} = 1+\log\frac{1}{1-t}$ is numerically given by $\tau \approx 0.682155$.
 One further gets $c_{\varphi} = \frac{\sqrt{1-\tau}}{\sqrt{8}} \approx 0.199325$ and $\sigma = \sqrt{1-\tau} \approx 0.563776$.
 Thus, $\frac{I_n}{c_{\varphi} n^{3/2}}$ converges in distribution to an Airy distributed random variable, and $\frac{\sqrt{n}}{n-j}I_{n,j}$ converges, for $n-j \ll \sqrt{n}$, in distribution to a Rayleigh distributed random variable with parameter $\sigma$.
\end{itemize}

\section{Proof of the results concerning the global behaviour\label{sec4}}

\subsection{Short overview of the proof}
We prove our result given in Theorem~\ref{thm:global} by using the method of moments, i.e., we show that the moments of $I_n$ converge
(after proper normalization) to the moments of the Airy distribution. Since this distribution
is uniquely determined by its moments, the convergence result then follows directly from the
theorem of Fr\'echet and Shohat \cite{Loeve_ProbabilityTheory}.
To start with, we do not study the random
variable $I_n$ directly, but consider a closely related random variable $\hI_n$. Using the tree decomposition as in \cite{Gessel-Sagan-Yeh_TreeInversions}, we then obtain a 
$q$-difference-differential equation for a suitably chosen generating function which encodes
the distribution of $\hI_n$. From this equation, we can ``pump'' the moments of $\hI_n$ 
using techniques from \cite{Flajolet-Poblete-Viola_LinearProbingHashing} and singularity analysis, and finally transfer our result to $I_n$.

\subsection{Introduction of $\hI_n$ and generating functions}
We let $\hT$ be the subset of $\T$ which consists exactly of those trees in which the 
root has label $1$. Obviously, the total weight of trees of size $n$ in $\hT$ is then given
by $\frac{T_n}{n}$. Also note that each tree in $\hT$ has the nice property that the root is not
part of any inversion. Hence, the total number of inversions can just be obtained by summing up
the contributions of the individual subtrees of the root. This fact will later be useful when we
translate a decomposition of the trees in $\hT$ to generating functions.

We let $\hI_n$ denote the number of inversions in a random tree of size $n$ in $\hT$, where each 
element of $\hT$ of size $n$ is chosen with probability proportional to its weight.

Furthermore, we introduce the generating function
\begin{equation}\label{eqn:F}
 F(z,q) := \sum_{n \geq 1} \sum_{k \geq 0} \PP{\hI_n = k} \frac{T_n}{n} q^k \frac{z^n}{n!}.
\end{equation}
Note that $n! [z^n q^k] F(z,q) = \PP{\hI_n = k} \frac{T_n}{n}$ is the total weight of all 
trees of size $n$ in $\hT$ which contain exactly $k$ inversions.
Moreover, observe that $\V F(z,q) = F(z,1)$ is just the exponential generating function of 
$\left(\frac{T_n}{n}\right)_{n \geq 1}$, and hence we have the relation 
\begin{equation}\label{eqn:F-T}
 \Z \DZ \V F(z,q) = T(z),
\end{equation}
which we will use frequently.
We further introduce the functions
\[
 f_r(z) := \V \DQ^r F(z,q),
\]
which are generating functions of the factorial moments $\EE{(\hI_n)^{\ff{r}}}$ of $\hI_n$, in
the sense that
\begin{equation}\label{eqn:f_r}
 f_r(z) = \sum_{n \geq 1} \EE{\hI_n^{\ff{r}}} \frac{T_n}{n} \frac{z^n}{n!}.
\end{equation}
Clearly, we can recover the $r$-th factorial moment of $\hI_n$ from \eqref{eqn:f_r} by
\[
 \EE{\hI_n^{\ff{r}}} = \frac{[z^n] f_r(z)}{[z^n] f_0(z)},
\]
but as we will see later, it is more convenient to use
\begin{equation}\label{eqn:fact-mom}
 \EE{\hI_n^{\ff{r}}} = \frac{[z^n] z f'_r(z)}{[z^n] zf'_0(z)} 
   = \frac{[z^n] z f'_r(z)}{[z^n] T(z)},
\end{equation}
where the second equality follows from \eqref{eqn:F-T}.

\subsection{The $q$-difference-differential equation for $F(z,q)$}
It turns out that $F(z,q)$ satisfies a certain equation involving a $q$-difference operator 
$\H$ which is very similar to the one Flajolet, Poblete and Viola used in 
\cite{Flajolet-Poblete-Viola_LinearProbingHashing} in their analysis of linear probing hashing.
In our case, we define $\H$ by
\[
 \H G(z,q) := \frac{G(z,q)-G(qz,q)}{1-q}.
\]
Using this, we get:
\begin{lemma}
 The function $F(z,q)$ defined by \eqref{eqn:F} satisfies
 \begin{equation}\label{eqn:F-HF}
  \DZ F(z,q) = \varphi(\H F(z,q)).
 \end{equation}
\end{lemma}
\begin{proof}
 This equation can be obtained by establishing mutually dependent recurrences for the total
 weights $T_{n,k} := \PP{I_n = k} T_n$ and $\hat{T}_{n,k} := \PP{\hI_n = k} \frac{T_n}{n}$
 of trees of size $n$ with $k$ inversions in $\T$ and $\hT$, respectively.
 Nevertheless, we confine ourselves to give a combinatorial argument at this point.
 
 In order to derive \eqref{eqn:F-HF}, we establish relations between $\T$ and $\hT$,
 which can be translated into functional equations for 
 $F(z,q) = \sum_{n \geq 1} \sum_{k \geq 0} \hat{T}_{n,k} q^k \frac{z^n}{n!}$ and
 \[
  T(z,q) := \sum_{n \geq 1} \sum_{k \geq 0} T_{n,k} q^k \frac{z^n}{n!}.
 \]
 For this purpose, we consider the sets $\T_n$ and $\hT_n$, which contain exactly the trees of size 
 $n$ of $\T$ and $\hT$, respectively. Clearly, $\T_n$ can be partitioned into $n$ disjoint 
 subsets $\T^{(1)}_n = \hT_n, \T^{(2)}_n \ldots, \T^{(n)}_n$, where $\T^{(j)}_n$ contains exactly 
 those trees in which the root is labelled by $j$. Now consider the bijective mapping between 
 $\hT_n$ and $\T^{(2)}_n$ which is obtained by just switching the labels $1$ and $2$ in each tree, 
 and leaving all other labels and the structure of each tree unchanged. Since this mapping does
 not alter the relative order of any pair of nodes except $(1,2)$, it clearly holds that each 
 tree of $\hT_n$ with $k$ inversions is mapped to a tree of $\T^{(2)}_n$ with $k+1$ inversions. 
 Repeating this argument, we see that each tree in $\hT_n$ with $k$ inversions can bijectively be 
 mapped to a tree with $k+j-1$ inversions in $\T^{(j)}_n$.
 This leads for the generating functions $F(z,q)$ and $T(z,q)$ to the equation
 \begin{equation}\label{eqn:T-HF}
  T(z,q) = \sum_{n \geq 1} \sum_{k \geq 0} \hat{T}_{n,k} 
  \underbrace{(1 + \ldots + q^{n-1})}_{\frac{1-q^n}{1-q}} q^k \frac{z^n}{n!} = \H F(z,q).
 \end{equation}
 Next, remember that $\T$ is defined by the formal
 equation $\T = \bigcirc * \varphi(\T)$, and that $\hT$ consists exactly of those trees of $\T$ in 
 which the root has label $1$. It thus follows that $\hT$ satisfies the formal equation
 \begin{equation}\label{eqn:hT-T}
  \hT = \nodeOne \,\times \varphi(\T).
 \end{equation}
 Due to the observation that the root node $\nodeOne$ of any tree in $\hT$ does not contribute to 
 the number of inversions, equation \eqref{eqn:hT-T} can be translated by an application of the 
 symbolic method to the differential equation
 \[
  \D_z F(z,q) = \varphi(T(z,q)).
 \]
 Using equation \eqref{eqn:T-HF}, we thus obtain \eqref{eqn:F-HF}.
\end{proof}

\subsection{Application of the pumping method}
In order to extract expressions for the functions $f_r(z)$ as defined in \eqref{eqn:f_r} from \eqref{eqn:F-HF}, we use the 
pumping method from \cite{Flajolet-Poblete-Viola_LinearProbingHashing}. 
This method basically rests on the idea of applying the operator $\V \DQ^r$ to the given 
functional equation involving $\H$, and using a "commutation rule" for the operators $\V \DQ^r$ 
and $\H$. 
Since our operator $\H$ is slightly different from the one in 
\cite{Flajolet-Poblete-Viola_LinearProbingHashing}, we will first establish the suitable 
commutation rule for our case. 
\begin{lemma}\label{lemma:commutation} The operator $\V \DQ^j \H $ satisfies the operator equation
\begin{equation}\label{eqn:commutation}
\V \DQ^j \H = \sum_{s=0}^j \binom{j}{s} \frac{1}{s+1} \Z^{s+1} \DZ^{s+1} \V \DQ^{j-s}.
\end{equation}
\end{lemma}
\begin{proof}
Since all occurring operators are linear, it suffices to show that the two sides of the equation 
coincide when applied to a function of the form $G(z,q) = q^k z^n$.
Remember that $\H (q^k z^n) = q^k (1 + q + \cdots + q^{n-1}) z^n$, and thus we have
\begin{eqnarray*}
\V \DQ^j \H (q^k z^n) &=& z^n \sum_{i=0}^{n-1} \V \DQ^j (q^k q^i)
= z^n \sum_{i=0}^{n-1} \sum_{s=0}^j \binom{j}{s} k^{\ff{j-s}} i^{\ff{s}} \\
&=& z^n \sum_{s=0}^j \binom{j}{s} k^{\ff{j-s}} s! \sum_{i=0}^{n-1} \binom{i}{s} 
= z^n \sum_{s=0}^j \binom{j}{s} k^{\ff{j-s}} s! \binom{n}{s+1}\\ 
&=& \sum_{s=0}^j \binom{j}{s} k^{\ff{j-s}} \frac{n^{\ff{s+1}} z^n}{s+1}  
= \sum_{s=0}^j \binom{j}{s} \frac{1}{s+1} \Z^{s+1} \DZ^{s+1} \V \DQ^{j-s} (q^k z^n).
\end{eqnarray*}
\end{proof}
Using \eqref{eqn:commutation}, we can now establish a recurrence for the derivatives $f'_r(z)$
of the factorial moment generating functions:
\begin{lemma}\label{lemma:f_r}
The factorial moment generating functions $f_r(z)$ satisfy, for $r \geq 1$,
\begin{multline}\label{eqn:recursion-general}
f'_r(z) = \frac{1}{1 - z \varphi'(T(z))} \left( \varphi'(T(z)) \sum_{t=1}^r \binom{r}{t} 
    \frac{1}{t+1} z^{t+1} f^{(t+1)}_{r-t} (z) \right.\\
\left. + \sum_{(k_1, \ldots, k_{r-1}) \in B_r} \frac{r!}{k_1! \cdots k_{r-1}!} 
    \varphi^{(k_1 + \ldots + k_{r-1})} (T(z)) \prod_{m=1}^{r-1} \left( \frac{1}{m!} 
    \sum_{s=0}^m \binom{m}{s} \frac{1}{s+1} z^{s+1} f^{(s+1)}_{m-s}(z) \right)^{k_m} \right),
\end{multline}
where $B_1 := \emptyset$, and
\[
 B_r := \set{(k_1, k_2, \ldots, k_{r-1}) \in \NN_0^{r-1} : k_1 + 2k_2 + \ldots + (r-1) k_{r-1} = r},
\]
for $r \geq 2$.
\end{lemma}

\begin{proof}
We apply $\V \DQ^r$ to \eqref{eqn:F-HF} and express $\DQ^r \varphi(\H F(z,q))$ using Fa\'a di 
Bruno's formula for higher derivatives of composite functions,
\[
\DQ^r f(g(q)) = \sum_{(k_1, \ldots, k_r) \in A_r} \frac{r!}{k_1! \cdots k_r!}
(\DQ^{(k_1+\ldots+k_r)} f)(g(q))
\prod_{m=1}^r\left(\frac{\DQ^m g(q)}{m!}\right)^{k_m},
\]
where $A_r := \set{(k_1, k_2, \ldots, k_r) \in \NN_0^r : k_1 + 2k_2 + \ldots + r k_r = r}$. 
We then obtain the claimed result by applying Lemma \ref{lemma:commutation}, solving for $f'_r(z)$ 
and using the fact that
\[
\V \varphi(\H F(z,q)) = \varphi(\V \H F(z,q)) = \varphi(\Z \DZ \V F(z,q)) = \varphi(T(z)),
\]
which follows from \eqref{eqn:commutation} and \eqref{eqn:F-T}.
\end{proof}

\subsection{Singularity analysis}
We now investigate the singular behaviour of the functions in \eqref{eqn:recursion-general} in 
order to compute the factorial moments $\EE{\hI_n^{\ff{r}}}$ asymptotically.
In the following, we carry out only the computations for the case $d=1$ (where $d$ is defined by \eqref{eqn:d} and thus gives the number of dominant singularities of the functions considered).
The general case runs completely analogous: when applying singularity analysis, one just has to 
take care of the contributions of all $d$ singularities and add them.

In a first step, we want to find an asymptotic formula for the expected value $\EE{\hI_n}$.
From Lemma \ref{lemma:f_r}, we have
\[
f'_1(z) = \frac{\varphi'(T(z)) \frac{z^2}{2} f''_0(z)}{1-z\varphi'(T(z))},
\]
and using the fact that
$f''_0(z) = T'(z) \varphi'(T(z)) = \frac{\varphi(T(z)) \varphi'(T(z))}{1-z\varphi'(T(z))}$, which
is easily obtained by differentiating \eqref{eqn:T} and \eqref{eqn:F-T}, we get
\[
f'_1(z) = \frac{1}{2} \frac{z (\varphi'(T(z)))^2 T(z)}{(1-z\varphi'(T(z)))^2}.
\]
Note that $z\varphi'(T(z)) \neq 1$ for $|z| \leq \rho$, which can be seen by 
differentiating \eqref{eqn:T}, and thus $f'_1(z)$ inherits the dominant singularity at 
$z = \rho$ from $T(z)$ and $\varphi'(T(z))$.
Using the expansions \eqref{eqn:T_exp} and \eqref{eqn:zphi_exp}, we find
\begin{equation}\label{eqn:f_1_exp}
\begin{split}
z f'_1(z) &= \frac{1}{2} \frac{\oneplusOO{(\rho - z)^{1/2}}^2  \tau \oneplusOO{(\rho - z)^{1/2}}}
    {2 \rho \tau \varphi''(\tau) \left( 1 - \frac{z}{\rho}\right) \oneplusOO{(\rho - z)^{1/2}}},\\
 &= \frac{1}{4 \rho \varphi''(\tau) \left( 1 - \frac{z}{\rho}\right)} 
    \oneplusOO{(\rho - z)^{1/2}}, \quad z \to \rho.
\end{split}
\end{equation} 
By applying basic singularity analysis, this immediately yields
\[
[z^n] z f'_1(z) = \frac{1}{4 \rho^{n+1} \varphi''(\tau)} \left( 1 + \OO{n^{-1/2}}\right).
\]
Now, using \eqref{eqn:fact-mom} and \eqref{eqn:T_n}, we get the expected value
\begin{eqnarray*}
\EE{\hI_n} &=& \frac{[z^n] z f'_1(z)}{[z^n] T(z)} = 
    \frac{1}{\rho} \sqrt{\frac{\pi}{8 \varphi(\tau) \varphi''(\tau)}} n^{3/2} \oneplusOO{n^{-1/2}}\\
 &=& c_{\varphi} \sqrt{\pi} n^{3/2} \oneplusOO{n^{-1/2}},
\end{eqnarray*}
where $c_{\varphi}$ is defined as in Theorem \ref{thm:global}.

We will now consider $f'_r(z)$ for general $r$. It turns out that all $f'_r(z)$ have a unique 
dominant singularity at $z = \rho$. 
The singular expansions around this point are given in the following lemma.
\begin{lemma}\label{lemma:f_r-singular_behaviour-general}
For $r \geq 1$, each $f'_r(z)$ has a unique dominant singularity at $z = \rho$, where the expansion
\begin{equation}\label{eqn:f_r_exp}
z f'_r(z) = c_\varphi^r \sqrt{\frac{\varphi(\tau)}{2 \varphi''(\tau)}} \frac{2 C_r}{\left( 1 - \frac{z}{\rho}\right)^{(3r-1)/2}} \left( 1 + \OO{(\rho-z)^{1/2}}\right), \quad z \to \rho,
\end{equation}
holds. Here, the constants $C_r$ are defined as in \eqref{eqn:C_r}.
\end{lemma}

\begin{proof}
One easily checks that in the case $r=1$ equation \eqref{eqn:f_r_exp} coincides with 
\eqref{eqn:f_1_exp}. For $r>1$ we proceed by induction, following the inductive
definition \eqref{eqn:C_r} of the constants $C_r$. 
So let $r > 1$ and assume that \eqref{eqn:f_r_exp} holds for all functions $f_j(z)$ with 
$1 \leq j < r$.
By the rules for singular differentiation \cite{Fill-Flajolet-Kapur_SingularityAnalysis} we
then also have the following singular expansions for the $k$-th derivatives $f^{(k)}_j$ of
the functions $f_j(z)$, for all $k \geq 1$ and $1 \leq j < r$:
\begin{equation}\label{eqn:zf_k_j}
z f^{(k)}_j(z) = c_\varphi^r \sqrt{\frac{\varphi(\tau)}{2 \varphi''(\tau)}} \frac{2 C_j \cdot \left( \frac{3j-1}{2} \right)^{\rf{k-1}}}{\rho^{k-1} \left( 1 - \frac{z}{\rho}\right)^{(3j-3+2k)/2}} \left( 1 + \OO{(\rho-z)^{1/2}}\right), \quad z \to \rho.
\end{equation}
From this one concludes that the dominant contributions in \eqref{eqn:recursion-general} can only 
arise from the terms corresponding to $t=1$ and $s=0$, i.e.,
\begin{multline*}
zf'_r(z) = \frac{z}{1 - z \varphi'(T(z))} \bigg( \varphi'(T(z)) \frac{r}{2} z^2 f''_{r-1}(z)\\
\mbox{} + \sum_{(k_1, \ldots, k_{r-1}) \in B_r} \frac{r!}{k_1! \cdots k_{r-1}!} \varphi^{(k_1 + \ldots + k_{r-1})} (T(z)) \prod_{m=1}^{r-1} \left( \frac{z f'_m(z)}{m!} \right)^{k_m} \bigg) \left( 1 + \OO{(\rho-z)^{1/2}} \right).
\end{multline*}
Now note that 
\[
\prod_{m=1}^{r-1} \left( \frac{1}{m!} z f'_m(z) \right)^{k_m} = \OO{\frac{1}{(\rho-z)^{(3r - (k_1 + \ldots + k_{r-1}))/2}}},
\]
hence the dominant terms in the remaining sum correspond to those $(k_1, \ldots, k_{r-1}) \in B_r$ 
with $k_1 + \ldots + k_{r-1} = 2$, and we thus get
\begin{eqnarray*}
zf'_r(z) &=& \frac{z}{1 - z \varphi'(T(z))} \bigg( \varphi'(T(z)) \frac{r}{2} z^2 f''_{r-1}(z)\\
& & \mbox{} + \sum_{s=1}^{r-1} \frac{r!}{2} \varphi'' (T(z)) z^2 \frac{f'_s(z)}{s!}  \frac{f'_{r-s}(z)}{(r-s)!} \bigg) \left( 1 + \OO{(\rho-z)^{1/2}} \right).
\end{eqnarray*}
Now, expanding the occurring functions using \eqref{eqn:phi_k}, \eqref{eqn:zphi_exp} and
\eqref{eqn:zf_k_j}, we obtain after some simplifications
\begin{eqnarray*}
zf'_r(z) &=& \frac{\rho}{\sqrt{2 \rho \tau \varphi''(\tau)} \sqrt{1 - \frac{z}{\rho}}} 
    \left( \frac{r}{2} c_{\varphi}^{r-1} \sqrt{\frac{\varphi(\tau)}{2 \varphi''(\tau)}} 
    \frac{2 C_{r-1} \cdot \frac{3(r-1)-1}{2}}{\rho \left( 1 - \frac{z}{\rho}\right)^{(3(r-1)+1)/2}}
    \right.\\
& & \quad \left. \mbox{} + \frac{1}{2} \sum_{s=1}^{r-1} \binom{r}{s} \varphi''(\tau) 
c_\varphi^{s} c_\varphi^{r-s} \left(\frac{\varphi(\tau)}{2 \varphi''(\tau)}\right) \frac{2 C_s \cdot 2 C_{r-s}}{\left( 1 - \frac{z}{\rho}\right)^{(3s-1)/2 + (3(r-s)-1)/2}}
 \right)\\
 & & \times \oneplusOO{(\rho-z)^{1/2}}\\ 
&=& c_\varphi^r \sqrt{\frac{\varphi(\tau)}{2 \varphi''(\tau)}} \frac{(3r-4)r C_{r-1} + \sum_{s=1}^{r-1} \binom{r}{s} C_s C_{r-s}}{\left( 1 - \frac{z}{\rho}\right)^{(3r-1)/2}} \left( 1 + \OO{(\rho-z)^{1/2}}\right),\\
&=& c_\varphi^r \sqrt{\frac{\varphi(\tau)}{2 \varphi''(\tau)}} \frac{2 C_r}{\left( 1 - \frac{z}{\rho}\right)^{(3r-1)/2}} \left( 1 + \OO{(\rho-z)^{1/2}}\right), \quad z \to \rho.
\end{eqnarray*}
\end{proof}
Lemma \ref{lemma:f_r-singular_behaviour-general} can now be used in order to compute the moments
of $\hI_n$ asymptotically:
\begin{lemma}\label{lem:hI_n}
The random variable $\hI_n$ satisfies
\[
\EE{\hI_n^{r}} = 
    \frac{2 \sqrt{\pi} c_\varphi^r n^{3r/2}}{\Gamma(\frac{3r-1}{2})} C_r \left(1 + \OO{n^{-1/2}}\right).
\] 
\end{lemma}

\begin{proof}
By singularity analysis, it follows from Lemma \ref{lemma:f_r-singular_behaviour-general} that
\[
    [z^n] z f'_r(z) = 2 c_\varphi^r \sqrt{\frac{\varphi(\tau)}{2 \varphi''(\tau)}} 
        \frac{n^{\frac{3r-1}{2} - 1}}{\Gamma(\frac{3r-1}{2})} C_r \oneplusOO{n^{-1/2}},
\]
and together with \eqref{eqn:fact-mom} and \eqref{eqn:T_n} this shows
\[
\EE{\hI_n^{\ff{r}}} = \frac{[z^n] z f'_r(z)}{[z^n] T(z)} = 
    \frac{2 \sqrt{\pi} c_\varphi^r n^{3r/2}}{\Gamma(\frac{3r-1}{2})} C_r \left(1 + \OO{n^{-1/2}}\right).
\]
Using the relation between the factorial moments and the ordinary moments of a random variable $Y$:
\begin{equation}\label{eqn:facordmom}
  \mathbb{E}(Y^{r}) = \sum_{\ell=0}^{r} \Stir{r}{\ell} \, \mathbb{E}(Y^{\underline{\ell}}),
\end{equation}
with $\Stir{r}{\ell}$ the Stirling numbers of second kind, we obtain that $\EE{\hI_n^{r}} = \EE{\hI_n^{\ff{r}}} + \OO{\EE{\hI_n^{\ff{r-1}}}}$, and
hence we get the desired result.
\end{proof}

\subsection{Transfer of the result to $I_n$} 
We now transfer the result for $\hI_n$ to the random variable $I_n$,
which counts the total number of inversions in a random tree of size $n$ of $\T$. In fact,
we prove that the moments of $\hI_n$ and $I_n$ coincide asymptotically:
\begin{lemma}\label{lem:I_n}
 The random variable $I_n$ satisfies
\begin{equation}\label{eqn:I_n}
\EE{I_n^{r}} = 
    \frac{2 \sqrt{\pi} c_\varphi^r n^{3r/2}}{\Gamma(\frac{3r-1}{2})} C_r \left(1 + \OO{n^{-1/2}}\right).
\end{equation}
\end{lemma}
\begin{proof}
 The relation between $\T$ and $\hT$ (compare equation \eqref{eqn:T-HF}) directly translates to the
 following relation between the moments of $I_n$ and $\hI_n$:
 \[
  \EE{I_n^{r}} = \frac{1}{n} \left( \EE{\hI_n^{r}} + \EE{(\hI_n+1)^{r}} + \ldots + \EE{(\hI_n+n-1)^{r}} \right).
 \]
 From this, one easily deduces that $\EE{I_n^{r}} = \EE{\hI_n^{r}} + \OO{n^{\frac{3r-1}{2}}}$,
 and hence \eqref{eqn:I_n} follows directly from Lemma \ref{lem:hI_n}.
\end{proof}

By comparing \eqref{eqn:I_n} with $\mu_r$ in Definition~\ref{def:airy}, we conclude that the 
moments of the normalized random variable $\frac{I_n}{c_{\varphi} n^{3/2}}$ converge to the 
moments of the Airy distribution. 
Due to Lemma \ref{lem:airy}, the convergence in distribution of $\frac{I_n}{c_{\varphi} n^{3/2}}$ 
to an Airy distributed random variable thus follows directly from the theorem of Fr\'echet and 
Shohat.

This finishes the proof of our result on the total number of inversions.

\section{Proof of the results concerning the local behaviour\label{sec5}}

\subsection{The generating functions approach}

A main ingredient in the proof of Theorem~\ref{thm:local} concerning the behaviour of the random variable $I_{n,j}$ is to introduce and study a suitable generating function for the probabilities 
$\mathbb{P}\{I_{n,j}=k\}$, which reflects in a simple way the recursive description of a tree as a root node and its subtrees.
It turns out that the following trivariate generating function is appropriate:
\begin{equation}\label{def:localgf}
  N(z,u,q) := \sum_{m \ge 0} \sum_{j \ge 1} \sum_{k \ge 0} \mathbb{P}\{I_{m+j,j}=k\} T_{m+j} \frac{z^{j-1}}{(j-1)!} \frac{u^{m}}{m!} q^{k}.
\end{equation}

\begin{prop}\label{prop:localgf}
  The generating function $N(z,u,q)$ is given by the following explicit formula:
  \begin{equation*}
    N(z,u,q) = \frac{\varphi(T(z+u))}{1-(z+uq)\varphi'(T(z+u))}.
  \end{equation*}
\end{prop}
\begin{proof}
  We will show the functional equation
  \begin{equation}\label{eqn:localgf}
    N(z,u,q) = \varphi(T(z+u)) + z \varphi'(T(z+u)) N(z,u,q) + uq \varphi'(T(z+u)) N(z,u,q),
  \end{equation}
  which is equivalent to the statement of Proposition~\ref{prop:localgf}.
  To do this we introduce specifically tricoloured trees: in each tree $T \in \mathcal{T}$ exactly one node is coloured red, all nodes with a label smaller than the red node are coloured white, whereas all nodes with a label larger than the red node are coloured black. Let us denote by $\mathcal{T}_{C}$ the family of all such tricoloured trees.
  Then in the generating function $N(z,u,q)$ the variable $z$ encodes the white nodes, the variable $u$ encodes the black nodes, whereas $q$ encodes the black ancestors of the red node, i.e.,
  \begin{equation*}
    N(z,u,q) = \sum_{T_{C} \in \mathcal{T}_{C}} w(T_{C}) \frac{z^{\sharp \; \text{white}}}{(\sharp \; \text{white})!} \frac{u^{\sharp \; \text{black}}}{(\sharp \; \text{black})!} 
    q^{\sharp \; \text{black ancestors of red}}.
  \end{equation*} 
  Since the black nodes as well as the white nodes are labelled it is appropriate to use a double exponential generating function.
  
  As auxiliary family we consider specifically bilabelled trees: the nodes in each tree $T \in \mathcal{T}$ are coloured black and white in a way such that each white node has a label smaller than any black node (i.e., all nodes up to a certain label are coloured white, whereas all remaining nodes are coloured black). Let us denote by $\mathcal{T}_{B}$ the set of all such bicoloured trees. The double exponential generating function of bicoloured trees,
  \begin{equation*}
    B(z,u) = \sum_{T_{B} \in \mathcal{T}_{B}} w(T_{B}) \frac{z^{\sharp \; \text{white}}}{(\sharp \; \text{white})!} \frac{u^{\sharp \; \text{black}}}{(\sharp \; \text{black})!},
  \end{equation*}
  can be computed easily. It holds:
  \begin{align}
    B(z,u) & = \sum_{n \ge 1 (0)} \sum_{T \in \mathcal{T}: |T|=n} w(T) \sum_{m=0}^{n} \frac{z^{n-m}}{(n-m)!} \frac{u^{m}}{m!}
    = \sum_{n \ge 0} \sum_{m=0}^{n} \frac{z^{n-m}}{(n-m)!} \frac{u^{m}}{m!} \sum_{T \in \mathcal{T}: |T|=n} w(T)\notag\\
    & = \sum_{n \ge 0} T_{n} \sum_{m=0}^{n} \frac{z^{n-m}}{(n-m)!} \frac{u^{m}}{m!}
    = \sum_{n \ge 0(1)} \frac{T_{n}}{n!} (z+u)^{n} = T(z+u).\label{eqn:Bzu}
  \end{align}
  
  Now we consider the decomposition of a tricoloured tree $T_{C} \in \mathcal{T}_{C}$ into the root node $\rt(T_{C})$ and its $\ell \ge 0$ subtrees $T_{1}, \dots, T_{\ell}$. 
  Thus the degree-weight of the root node is given by $\varphi_{\ell}$.
  Three cases may occur.
  \begin{itemize}
    \item[$(i)$] The root node is the red node. Then the red node does not have black ancestors and all of the subtrees $T_{1}, \dots, T_{\ell}$ are, after order preserving relabellings,
    specifically bicoloured trees, i.e., elements of $\mathcal{T}_{B}$.
    \item[$(ii)$] The root node is a white node. Then the red node is contained in one of the $\ell$ subtrees; let us assume it is $T_{s}$.
    After an order preserving relabelling this subtree is itself an element of $\mathcal{T}_{C}$, whereas all remaining subtrees are, after order preserving relabellings, elements of $\mathcal{T}_{B}$.
    Moreover, the number of black ancestors of the red node in $T_{C}$ is the same as the number of black ancestors of the red node in the subtree $T_{s}$.
    \item[$(iii)$] The root node is a black node. Again the red node is contained in of the $\ell$ subtrees; let us assume it is $T_{s}$.
    After an order preserving relabelling this subtree is an element of $\mathcal{T}_{C}$, whereas all remaining subtrees are, after order preserving relabellings, elements of $\mathcal{T}_{B}$.
    But in this case the number of black ancestors of the red node in $T_{C}$ is one more than the number of black ancestors of the red node in the subtree $T_{s}$.
  \end{itemize}
  Considering all tricoloured trees of $\mathcal{T}_{C}$ and taking into account \eqref{eqn:Bzu}
  the above decomposition leads to the stated equation~\eqref{eqn:localgf} for $N(z,u,q)$:
  \begin{align*}
     N(z,u,q) & = \sum_{\ell \ge 0} \varphi_{\ell} \big(T(z+u)\big)^{\ell} + z \sum_{\ell \ge 0} \ell \varphi_{\ell} \big(T(z+u)\big)^{\ell-1} N(z,u,q) \\
     & \quad \mbox{} + uq \sum_{\ell \ge 0} \ell \varphi_{\ell} \big(T(z+u)\big)^{\ell-1} N(z,u,q) \\
     & = \varphi(T(z+u)) + z \varphi'(T(z+u)) N(z,u,q) + uq \varphi'(T(z+u)) N(z,u,q).
  \end{align*}
\end{proof}

\subsection{Computing the factorial moments}

Starting with the explicit formula for the trivariate generating function $N(z,u,q)$
given in Proposition~\ref{prop:localgf} we will compute the $r$-th factorial moments
of $I_{n,j}$. According to the definition~\eqref{def:localgf} of $N(z,u,q)$ one obtains:
\begin{align*}
  \mathbb{E}(I_{n,j}^{\underline{r}}) & = \frac{(j-1)! (n-j)!}{T_{n}}
  [z^{j-1} u^{n-j}] \V D_{q}^{r} N(z,u,q)\\
  & = \frac{(j-1)! \, (n-j)! \, r!}{T_{n}} [z^{j-1} u^{n-j-r}]
  \frac{\big(\varphi'(T(z+u))\big)^{r} \varphi(T(z+u))}{\big(1-(z+u)\varphi'(T(z+u))\big)^{r+1}}.
\end{align*}
Since for any power series $g(x)$ it holds:
\begin{equation*}
  [z^{a} u^{b}] g(z+u) = \binom{a+b}{a} [z^{a+b}] g(z),
\end{equation*}
one further obtains the following expression, which will be the starting point for our asymptotic
considerations:
\begin{align}
  \mathbb{E}(I_{n,j}^{\underline{r}}) & = \frac{(j-1)! \, (n-j)! \, r!}{T_{n}} \binom{n-r-1}{j-1} [z^{n-r-1}] \frac{\big(\varphi'(T(z))\big)^{r} \varphi(T(z))}{\big(1-z \varphi'(T(z))\big)^{r+1}} \notag\\
  & = \frac{(j-1)! \, (n-j)! \, r!}{T_{n}} \binom{n-r-1}{j-1}
  [z^{n}] \frac{\big(z \varphi'(T(z))\big)^{r} T(z)}{\big(1-z \varphi'(T(z))\big)^{r+1}}.
  \label{eqn:Injr}
\end{align}

In order to evaluate $\mathbb{E}(I_{n,j}^{\underline{r}})$ asymptotically we use the 
local expansions \eqref{eqn:T_exp} and \eqref{eqn:zphi_exp} and apply singularity analysis.
Again for simplicity in presentation we will only carry out the computations for the
case that the functions involved have $d=1$ dominant singularities (see Subsection~\ref{ssec22}); for $d > 1$ one just has to add the contributions of all these singularities.

We obtain then (for $r$ arbitrary, but fixed):
\begin{align*}
  & [z^{n}] \frac{\big(z \varphi'(T(z))\big)^{r} T(z)}{\big(1-z \varphi'(T(z))\big)^{r+1}}
  = [z^{n}] \frac{\big(1+\mathcal{O}\big(\sqrt{1-\frac{z}{\rho}}\big)\big)^{r} 
  \big(\tau + \mathcal{O}\big(\sqrt{1-\frac{z}{\rho}}\big)\big)}{\big(\sqrt{2\rho\tau\varphi''(\tau)}
  \sqrt{1-\frac{z}{\rho}} + \mathcal{O}(1-\frac{z}{\rho})\big)^{r+1}}\\
  & = [z^{n}] \frac{\tau}{(2 \rho \tau \varphi''(\tau))^{\frac{r+1}{2}}}
  \frac{1}{(1-\frac{z}{\rho})^{\frac{r+1}{2}}} \cdot \big(1+\mathcal{O}\big(\textstyle{\sqrt{1-\frac{z}{\rho}}}\big)\big)
  = \frac{\tau \, n^{\frac{r-1}{2}}}{(2 \rho \tau \varphi''(\tau))^{\frac{r+1}{2}}
  \rho^{n} \Gamma(\frac{r+1}{2})} \cdot \big(1+\mathcal{O}\big(n^{-\frac{1}{2}}\big)\big).
\end{align*}

Together with the asymptotic formula for $T_{n}$ given in \eqref{eqn:T_n}, we obtain 
from \eqref{eqn:Injr} after simple computations:
\begin{align*}
  \mathbb{E}(I_{n,j}^{\underline{r}}) & = \frac{(n-j)! r! (n-r-1)! \sqrt{2 \pi \varphi''(\tau)} n^{\frac{3}{2}} \tau n^{\frac{r-1}{2}}}
  {(n-r-j)! n! \sqrt{\varphi(\tau)} (2 \rho \tau \varphi''(\tau))^{\frac{r+1}{2}} \Gamma(\frac{r+1}{2})} \cdot \big(1+\mathcal{O}\big(n^{-\frac{1}{2}}\big)\big)\\
  & = \frac{r! \sqrt{\pi}}{(2\rho \tau \varphi''(\tau))^{\frac{r}{2}} \Gamma(\frac{r+1}{2})} \frac{(n-j)^{\underline{r}}}{n^{\frac{r}{2}}}
  \cdot \big(1+\mathcal{O}\big(n^{-\frac{1}{2}}\big)\big).
\end{align*}
Using the duplication formula for the Gamma-function:
\begin{equation*}
  \Gamma\big(\frac{r+1}{2}\big) \Gamma\big(\frac{r}{2}+1\big) = \frac{r! \sqrt{\pi}}{2^{r}},
\end{equation*}
we obtain the following expansion of the $r$-th factorial moment of $I_{n,j}$, which holds uniformly for all $1 \le j \le n$:
\begin{equation}\label{eqn:Inj_as}
  \mathbb{E}(I_{n,j}^{\underline{r}}) = \frac{\Gamma(\frac{r}{2}+1) 2^{\frac{r}{2}}}{(\rho \tau \varphi''(\tau))^{\frac{r}{2}}}
  \frac{(n-j)^{\underline{r}}}{n^{\frac{r}{2}}}
  \cdot \big(1+\mathcal{O}\big(n^{-\frac{1}{2}}\big)\big).
\end{equation}

\subsection{Limiting distributions by applying the method of moments}
The asymptotic behaviour of the moments of $I_{n,j}$ depending on the growth of $j=j(n)$ can be obtained easily from the uniform expansion \eqref{eqn:Inj_as}. An application of the method of moments
shows then the limiting distribution results stated in Theorem~\ref{thm:local}.

\subsubsection{Region $n-j \gg \sqrt{n}$}
For this region it holds
\begin{equation*}
  \frac{(n-j)^{\underline{r}}}{n^{\frac{r}{2}}} = \frac{(n-j)^{r}}{n^{\frac{r}{2}}} \cdot \big(1+\mathcal{O}\big(n^{-\frac{1}{2}}\big)\big),
\end{equation*}
which implies the following expansion for the factorial moments:
\begin{equation}\label{eqn:facmomexp}
  \mathbb{E}(I_{n,j}^{\underline{r}}) = \frac{2^{\frac{r}{2}} \Gamma(\frac{r}{2}+1)}{(\rho \tau \varphi''(\tau))^{\frac{r}{2}}} \frac{(n-j)^{r}}{n^{\frac{r}{2}}}
  \cdot \big(1+\mathcal{O}\big(n^{-\frac{1}{2}}\big)\big).
\end{equation}
Together with equation~\eqref{eqn:facordmom} connecting the factorial and the ordinary moments we obtain the following asymptotic expansion for the $r$-th moments of $I_{n,j}$:
\begin{equation*}
  \mathbb{E}(I_{n,j}^{r}) = \frac{2^{\frac{r}{2}} \Gamma(\frac{r}{2}+1)}{(\rho \tau \varphi''(\tau))^{\frac{r}{2}}} \frac{(n-j)^{r}}{n^{\frac{r}{2}}}
  \cdot \Big(1+\mathcal{O}\big(\frac{\sqrt{n}}{n-j}\big)\Big).
\end{equation*}
Thus we obtain, for each $r$ fixed and $n \to \infty$:
\begin{equation*}
   \mathbb{E}\Big(\Big(\frac{\sqrt{n}}{n-j} I_{n,j}\Big)^{r}\Big) \to \Big(\frac{1}{\sqrt{\rho \tau \varphi''(\tau)}}\Big)^{r} 2^{\frac{r}{2}} \, \Gamma\big(\frac{r}{2}+1\big),
\end{equation*}
i.e., the moments of $\frac{\sqrt{n}}{n-j} I_{n,j}$ converge to the moments of a Rayleigh distributed random variable with parameter $\sigma = \frac{1}{\sqrt{\rho \tau \varphi''(\tau)}}$.
An application of the theorem of Fr\'echet and Shohat shows then the corresponding limiting distribution result of Theorem~\ref{thm:local}.

\subsubsection{Region $n-j \sim \alpha \sqrt{n}$, $\alpha \in \mathbb{R}^{+}$}
Also for this region the asymptotic expansion \eqref{eqn:facmomexp} of the $r$-th factorial moments computed above holds and one gets further:
\begin{equation}
   \mathbb{E}(I_{n,j}^{\underline{r}}) \to \frac{2^{\frac{r}{2}} \Gamma(\frac{r}{2}+1)}{(\rho \tau \varphi''(\tau))^{\frac{r}{2}}} \, \alpha^{r}
   = \Big(\frac{\alpha}{\sqrt{\rho \tau \varphi''(\tau)}}\Big)^{r} 2^{\frac{r}{2}} \, \Gamma\big(\frac{r}{2}+1\big).
\end{equation}

To continue we require the following lemma.
\begin{lemma}
  Let $Y_{\gamma}$, with $\gamma > 0$, be a discrete random variable with distribution
  \begin{equation*}
    \mathbb{P}\{Y_{\gamma}=k\} = \frac{\gamma^{k}}{k!} \int_{0}^{\infty} x^{k+1} e^{-\frac{x^{2}}{2} - \gamma x} dx, \quad \text{for} \; k \ge 0.
  \end{equation*}
  
  Then it holds that the $r$-th factorial moments of $Y_{\gamma}$ are given as follows:
  \begin{equation*}
    \mathbb{E}(Y_{\gamma}^{\underline{r}}) = \gamma^{r} \, 2^{\frac{r}{2}} \, \Gamma\big(\frac{r}{2}+1\big).
  \end{equation*}
  Moreover, the distribution of $Y_{\gamma}$ is uniquely defined by its sequence of moments.
\end{lemma}
\begin{proof}
  For $r \ge 0$ we get (the case $r=0$ shows that the probabilities sum up to $1$, i.e., they define indeed a distribution):
  \begin{align*}
    \mathbb{E}(Y_{\gamma}^{\underline{r}}) & = \sum_{k \ge 0} k^{\underline{r}} \, \frac{\gamma^{k}}{k!} \int_{0}^{\infty} x^{k+1} e^{-\frac{x^{2}}{2} - \gamma x} dx
    = \int_{0}^{\infty} e^{-\frac{x^{2}}{2} - \gamma x} \gamma^{r} x^{r+1} \sum_{k \ge r} \frac{(\gamma x)^{k-r}}{(k-r)!} dx\\
    & = \int_{0}^{\infty} e^{-\frac{x^{2}}{2} - \gamma x} \gamma^{r} x^{r+1} e^{\gamma x} dx
    = \gamma^{r} \int_{0}^{\infty} x^{r+1} e^{-\frac{x^{2}}{2}} dx = 2^{\frac{r}{2}} \gamma^{r} \int_{0}^{\infty} u^{\frac{r}{2}} e^{-u} du\\
    & = \gamma^{r} \, 2^{\frac{r}{2}} \, \Gamma\big(\frac{r}{2}+1\big).
  \end{align*}
  
  To show that the sequence of moments uniquely characterizes the distribution we consider the moment generating function $F(s) := \mathbb{E}(e^{s Y_{\gamma}})$ of $Y_{\gamma}$.
  It can be shown easily that $F(s)$ is given by the following expression:
  \begin{equation*}
    F(s) = \sum_{k \ge 0} \mathbb{P}\{Y_{\gamma} = k\} e^{k s} = \int_{0}^{\infty} x e^{-\frac{x^{2}}{2}-\beta x} dx, \quad \text{with} \enspace \beta = \gamma (1-e^{s}).
  \end{equation*}
  Thus the moment generating function $F(s)$ exists in a real neighbourhood of $s=0$ (actually it exists for all real $s$), which implies that the corresponding distribution is uniquely defined by its moments.
\end{proof}
Since the $r$-th factorial (and thus also ordinary) moments of $I_{n,j}$ converge to the corresponding moments of $Y_{\gamma}$, with $\gamma = \frac{\alpha}{\sqrt{\rho \tau \varphi''(\tau)}}$,
an application of the theorem of Fr\'echet and Shohat shows also for this case the limiting distribution results stated in Theorem~\ref{thm:local}.

\subsubsection{Region $n-j \ll \sqrt{n}$}

From \eqref{eqn:Inj_as} one easily gets that $\mathbb{E}(I_{n,j}^{r}) \to 0$, for $r \ge 1$, which, by an application of the theorem of Fr\'echet and Shohat shows $I_{n,j} \convd 0$ as stated in the corresponding part of Theorem~\ref{thm:local}.

\subsection{Explicit formulas for probabilities}

For some particular tree families it is possible to obtain explicit formulas for the probabilities $\mathbb{P}\{I_{n,j}=k\}$ by extracting coefficients from the trivariate generating function
$N(z,u,q)$ as given in \eqref{eqn:localgf}.
E.g., for ordered and unordered trees ($\varphi(t) = \frac{1}{1-t}$ and $\varphi(t) = e^{t}$, respectively) the generating function $N(z,u,q)$ is given by the following expressions:
\begin{align*}
   N(z,u,q) & = \frac{1-T(z+u)}{(1-T(z+u))^{2} - z - uq}, & \enspace \text{with} \enspace T(z) & = \frac{z}{1-T(z)} \quad \text{(ordered trees)},\\
   N(z,u,q) & = \frac{e^{T(z+u)}}{1-(z+uq)e^{T(z+u)}}, & \enspace \text{with} \enspace T(z) & = z \, e^{T(z)} \quad \text{(unordered trees)}.
\end{align*}
We omit here the necessary computations for extracting coefficients of $N(z,u,q)$ in order to obtain the required probabilities via
\begin{equation*}
  \mathbb{P}\{I_{n,j}=k\} = \frac{(j-1)! (n-j)!}{n!} \frac{[z^{j-1} u^{n-j} q^{k}] N(z,u,q)}{[z^{n}] T(z)},
\end{equation*}
but we stated the corresponding results in Subsection~\ref{ssec34} as formulas \eqref{eqn:InjOrd} and \eqref{eqn:InjUnord}.

 
\bibliography{Inversions}{}
\bibliographystyle{plain}

\end{document}